\documentclass[a4paper,reqno]{amsart}
\usepackage[english]{babel}
\usepackage{amsmath,amsthm,amssymb,amsfonts}

\usepackage{todonotes}
\usepackage{enumitem}
\usepackage{hyperref}
\usepackage{mathtools}
\usepackage[T1]{fontenc}
\usepackage[utf8]{inputenc}
\hypersetup{
 colorlinks,
 linkcolor={blue!90!black},
 citecolor={red!80!black},
 urlcolor={blue!50!black}
}
\usepackage{tikz}
\usepackage{tikz-cd}
\usepackage{graphicx}
\usepackage[all,cmtip]{xy}
\usepackage{mleftright}
\usepackage{booktabs}
\usepackage{cite}
\theoremstyle{theorem}
\newtheorem{theorem}{\sc \textbf{Theorem}}[section]
\newtheorem{proposition}[theorem]{\sc \textbf{Proposition}}
\newtheorem{corollary}[theorem]{\sc \textbf{Corollary}}
\newtheorem{lemma}[theorem]{\sc \textbf{Lemma}}

\renewcommand{\approx}{ \asymp}

\theoremstyle{remark}

\newtheorem{remark}[theorem]{\sc \textbf{Remark}}

\renewcommand{\theequation}{\thesection.\arabic{equation}}
\numberwithin{equation}{section}

\newcommand{\N}{\mathbb{N}}
\newcommand{\Z}{\mathbb{Z}}

\def\eps{\varepsilon}

\def\e{\mathrm{e}}

\def\tr{\operatorname{tr}}

\def\supp{\operatorname{supp}}

\newcommand{\dd}{d}

\usepackage{mathrsfs}

\renewcommand{\leq}{\leqslant}
\renewcommand{\geq}{\geqslant}

\usepackage{pdfpages}

\newcommand{\Ss}{\mathcal{S}}

\makeatletter
\@namedef{subjclassname@1991}{Mathematics Subject Classification}

\newcommand\numberthis{\addtocounter{equation}{1}\tag{\theequation}}

\graphicspath{{images/}}

\author[T.\ Bruno]{Tommaso Bruno}
\address{Dipartimento di Matematica, Università degli Studi di Genova\\ Via Dodecaneso 35, 16146 Genova, Italy}
\email{brunot@dima.unige.it}

\author[J.T.\ van Velthoven]{Jordy Timo van Velthoven}
\address{Delft University of Technology, Mekelweg 4, Building 36, 2628 CD Delft, the Netherlands.}
\curraddr{Faculty of Mathematics, University of Vienna, Oskar-Morgenstern-Platz 1, 1090 Vienna, Austria}

\email{jordy-timo.van-velthoven@univie.ac.at}

	\thanks{{\em Math Subject Classification} 22E25, 22E30, 42B35.}
	\thanks{The first named author was partially supported by the 2022 INdAM--GNAMPA grant {\em Generalized Laplacians on continuous and discrete structures} (CUP\_E55F22000270001) and the 2024 INdAM--GNAMPA grant {\em $L^{p}$ estimates for singular integrals in nondoubling settings} (CUP\_E53C23001670001). For the second named author, this research was funded in whole or in part by the Austrian Science Fund (FWF): 10.55776/J4555.}

\begin{document}
\title[Hardy spaces and dilations]{Hardy spaces and dilations  on \\ homogeneous groups}
	\maketitle
\vspace{-0.5cm}
\begin{abstract}
On a homogeneous group, we characterize the one-parameter groups of dilations whose associated Hardy spaces in the sense of Folland and Stein are the same.
\end{abstract}

\section{Introduction}
Let $G$ be a homogeneous group, i.e., a connected, simply connected nilpotent Lie group whose Lie algebra $\mathfrak{g}$ admits automorphic dilations
\[
\delta_r^A = \exp(\ln(r) A), \qquad r>0,
\]
for a diagonalizable matrix $A \in \mathrm{GL}(\mathfrak{g})$ with positive eigenvalues. As $G$ is simply connected and nilpotent, its exponential map is a global diffeomorphism, and the automorphisms $\delta_r^A$ induce automorphisms of $G$ which we still denote by $\delta_r^A$.

Following Folland and Stein~\cite{FollandStein}, we consider Hardy spaces associated to the dilations $(\delta_r^A)_{r>0}$ on $G$ as follows. Given a Schwartz function $\phi \in \mathcal{S}$ on $G$, the associated radial maximal function of a tempered distribution $f \in \mathcal{S}'$ is
\[
M^0_{\phi,A} f = \sup\nolimits_{r > 0} r^{\tr(A)} | f \ast  (\phi \circ \delta_r^A)|,
\]
where $\ast$ denotes the convolution on $G$. Upon fixing a commutative approximate identity $\phi$ (see~\cite{glowacki1986stable, dziubanski1992remark}), the Hardy space $H^p_A$, with $p \in (0,1]$, is the space
\[
H^p_A = \{f \in \mathcal{S}' \colon M_{\phi,A}^0 f \in L^p  \}
\]
endowed with the quasi-norm $ f \mapsto \|M_{\phi,A}^0 f\|_{p}^{p} $. It is well known that there are several other equivalent definitions, see, e.g.,~\cite{christ1984singular, glowacki1987inversion, sato2020hardy, FollandStein}, but we leave further discussions on such choice to a later stage, cf.~Remark~\ref{rem:FS} below.

\smallskip

The aim of this paper is to characterize those dilation matrices $A, B \in \mathrm{GL}(\mathfrak{g})$ as above which induce, via the associated dilations $(\delta_r^A)_{r>0}$ and $(\delta_r^B)_{r>0}$ respectively, the same Hardy spaces $H^p_A$ and $H^p_B$. Our main result is the following theorem.

\begin{theorem} \label{thm:main}
$H^p_A = H^p_B$ for some (equivalently, all) $p \in (0,1]$ if and only if $A = c B$ for some $c > 0$.
\end{theorem}

Notice that we do not assume that the spaces have equivalent quasi-norms, but just being the same as sets. We also emphasize that $H^{p}_{A} = H^{p}_{B}$ is in general \emph{not} equivalent to the homogeneous quasi-norms on $G$ induced by $A$ and $B$ being equivalent; cf.~Proposition~\ref{prop:equivalent_norm} below.

The problem of characterizing the dilations which give rise to the same function spaces looks rather natural. In the case when $G$ is abelian, that is, when $G$ is some Euclidean space $\mathbb{R}^n$, this has already been studied for Hardy spaces associated to anisotropic or parabolic dilations~\cite{bownik2003anisotropic, bownik2020pde, calderon1977atomic, calderon1977parabolic, calderon1975parabolic} in~\cite{bownik2003anisotropic}.  More recently, similar problems have been investigated for Besov and Triebel--Lizorkin spaces in,  e.g., \cite{cheshmavar2020classification, fuehr2022classifying, koppensteiner2022classification}.

If $G$ is abelian, Theorem~\ref{thm:main} may be obtained from a combination of results in~\cite{bownik2003anisotropic, bownik2020pde} on Hardy spaces defined by  expansive dilation matrices.  The novelty of Theorem~\ref{thm:main} is that it is the first instance of such results on noncommutative groups. Our approach to the problem is strongly influenced by the aforementioned papers, in particular by Bownik's~\cite{bownik2003anisotropic}. Nevertheless, the noncommutative setting requires a number of nontrivial modifications which we shall discuss along the way.
It finally goes without saying that Folland and Stein's book~\cite{FollandStein} plays a key role in the paper, too.

\smallskip

The structure of the paper is as follows. In the following Section~\ref{sec:norms} we characterize those matrices $A$ and $B$ whose induced homogeneous norms on $G$ are equivalent: we show that this happens if and only if $A=B$. In Section~\ref{sec:hardy} we introduce Hardy spaces on $G$ and describe equivalent characterizations of their semi-norms in terms of atomic decompositions and grand maximal functions. In the final Section~\ref{sec:equivhardy} we prove Theorem~\ref{thm:main} and discuss an analogous result for $BMO$ spaces.

\subsection*{Setting and notation} All throughout, $G$ denotes a homogeneous group with identity $e$ and Lie algebra $\mathfrak{g}$. The dimension of $\mathfrak{g}$, whence that of $G$, will be denoted by $n$. We shall say that a matrix $A \in \mathrm{GL}(\mathfrak{g})$ is \emph{admissible} if it is diagonalizable, has positive eigenvalues and the matrix exponential $\exp(A \ln(r))$, $r>0$, is an automorphism of $\mathfrak{g}$. Given such a matrix $A$ and $r>0$, we denote by $\delta_{r}^{A}$ both the automorphisms $\exp(A \ln(r))$ of $\mathfrak{g}$ and the corresponding group automorphisms of $G$ given by $\exp_G \circ \; \delta_r^A \circ \exp_G^{-1}$, where $\exp_G \colon \mathfrak{g} \to G$ is the exponential map of $G$.
 
Given two functions $f, g \colon X \to [0, \infty)$ on a set $X$, we write $f \lesssim g$ if there exists $C > 0$ such that $f(x) \leq C g(x)$ for all $x \in X$. The notation $f \asymp g$ will be used whenever $f \lesssim g$ and $g \lesssim f$.

\section{Equivalence of homogeneous quasi-norms}\label{sec:norms}
Given an admissible matrix $A$, a homogeneous quasi-norm associated with $A$ (equivalently, with the family of dilations $(\delta_{r}^{A})_{r>0}$) is a continuous function $\rho_A \colon G\to [0,\infty)$ which is smooth in $G\setminus \{e\}$ and satisfies, for $x\in G$,
\begin{enumerate}
\item[(1)] $\rho_{A}(x^{-1}) = \rho_{A}(x)$,
\item[(2)] $\rho_{A}(\delta_{r}^{A}x) = r\rho_{A}(x)$,
\item[(3)] $\rho_{A}(x)=0$ if and only if $x=e$.
\end{enumerate}
Homogeneous quasi-norms do exist for any given family of dilations $(\delta^A_r)_{r>0}$, and any two such quasi-norms $\rho_A, \rho_A'$ are mutually equivalent, in the sense that $\rho_A \approx \rho'_A $; see \cite[p. 8]{FollandStein} and \cite[Proposition 3.1.35]{fischer2016quantization} for proofs of both facts. In addition, for all semi-norms $\rho_{A}$ there exists $C>0$ such that
\begin{align} \label{eq:triangle}
 \rho_A(xy) \leq C(\rho_A (x) + \rho_A(y))
\end{align}
for all $x,y \in G$; see \cite[p.11]{FollandStein} and \cite[Proposition 3.1.38]{fischer2016quantization}.

Given $x_{0}\in G$ and $r>0$, the ball associated to a homogeneous norm $\rho_{A}$ centered at $x_{0}$ with radius $r$ is
\[
 \mathcal{B}^{A}(x_0, r) \coloneqq \{ x\in G \colon \rho_{A}(x_0^{-1} x)<r\}.
\]
With such definition, we also have
\begin{equation}\label{dilationball}
\mathcal{B}^{A}(x_0, r) = x_0 \mathcal{B}^A (e, r), \qquad \mathcal{B}^{A}(e, r) = \delta_{r}^{A}\mathcal{B}^{A}(e, 1).
\end{equation}
If $\lambda$ is the Lebesgue measure on $\mathfrak{g}$, we define the associated Haar measure $\mu$ on $G$ by $\mu = \lambda \circ \exp_{G}^{-1}$. Then, for all measurable subsets $E$ of $G$,
\begin{equation}\label{measuredilation}
\mu(\delta_{r}^{A}E) = r^{\tr(A)} \mu(E).
\end{equation}
In particular, $\mu(\mathcal{B}^{A}(x_{0},r)) = r^{\tr(A)} \mu(\mathcal{B}^{A}(e,1))$. In view of~\eqref{measuredilation}, the trace of $A$ is often called the \emph{homogeneous dimension} of $G$ (with respect to $A$).

In addition to homogeneous quasi-norms, we will also make use of a function on $G$ that is homogeneous with respect to dilations by multiples of the identity matrix.
For this, endow $\mathfrak{g}$ with an orthonormal basis $\{Y_{1}, \dots, Y_{n}\}$ and let $\| \cdot \|$ be the associated Euclidean norm. Then extend it to a function on $G$ by means of the exponential map, i.e., (with a slight abuse) $\| x \| = \| \exp_G^{-1} x\|$ for $x\in G$. Observe that $\| x^{-1}\| = \|x\|$, though $\| \cdot \|$ \emph{is not} a norm nor a quasi-norm on $G$ unless $G$ is abelian. We denote the ``ball'' of center $x_{0}$ and radius $r$ with respect to $\| \cdot \|$ simply by
\[
\mathcal{B}(x_{0},r) := \{ x\in G \colon \|x_{0}^{-1}x\|<r\} =  x_{0}\mathcal{B}(e,r).
\]
The function $\| \cdot \|$ on $G$ is homogeneous with respect to the classical (Euclidean) dilations $\delta^I_{t} := \exp_G \; \circ \exp(I \ln(t)) \circ \exp_G^{-1}$, $t>0$, namely $\|\delta^I_{t} x\|=t \|x\|$ for $x\in G$. However, we remark that such dilations are automorphisms of $G$ if and only if $G$ is abelian. More generally, we shall write
\[
\delta_{t}^{\Lambda} := \exp_G \circ \exp( \ln(t) \Lambda) \circ \exp_G^{-1}
\]
for $t > 0$ and general $\Lambda \in \mathrm{GL}(\mathfrak{g})$, which do not need to be automorphisms. Since the identity $I$ and its multiples commute with all matrices, the dilations $\delta^I_{t}$ commute with any other dilation $\delta^{\Lambda}_r$: for any $x\in G$, then,
\begin{equation}\label{euclnormcomm}
\delta_{t}^I \delta_{r}^{\Lambda} x = \delta_{r}^{\Lambda} \delta_{t}^I x, \quad r, t > 0.
\end{equation}
Given $\Lambda \in \mathrm{GL}(\mathfrak{g})$, we shall denote by $\| \Lambda \|_{\mathrm{GL}(\mathfrak{g})}$ its operator norm associated to the Euclidean norm $\| \cdot \|$ on $\mathfrak{g}$. Observe that $\delta_{1}^{\Lambda}$ is the identity map for all $\Lambda \in \mathrm{GL}(\mathfrak{g})$. 

 Given two admissible matrices $A$ and $B$, we say that two associated quasi-norms $\rho_{A}$ and $\rho_{B}$ are \emph{equivalent} if $\rho_{A} \approx \rho_{B}$, namely (we recall it for future use) if there exists a constant $C>0$ such that
\begin{equation}\label{equivnorms}
C^{-1}\rho_{B}(x) \leq \rho_{A}(x) \leq C \rho_{B}(x) \qquad
\end{equation}
for all $x \in G$. 

In the following proposition we show that the equivalence of homogeneous quasi-norms is a rather rigid condition; cf.~\cite[Lemma 10.2]{bownik2003anisotropic}.

\begin{proposition} \label{prop:equivalent_norm}
Let $A, B \in \mathrm{GL}(\mathfrak{g})$ be admissible matrices and $\rho_A, \rho_B$ associated quasi-norms. Then $\rho_A \asymp \rho_B$ if and only if $A=B$.
\end{proposition}
\begin{proof}
Since all homogeneous quasi-norms associated to an admissible matrix are equivalent, it follows that $\rho_A \asymp \rho_B$ for any choice of $\rho_{A}$ and $\rho_{B}$ whenever $A=B$. As for the converse, since $A, B$ are admissible, their exponentials $\exp(A)$ and $\exp(B)$ have only strictly positive eigenvalues, and thus admit a unique logarithm, see, e.g., \cite[Theorem 1.31]{higham2008functions}. Consequently, $\exp(A) = \exp(B)$ if and only if $A = B$, and it is then enough to prove that if $\rho_A \asymp \rho_B$ then $\exp(A)=\exp(B)$.

If~\eqref{equivnorms} holds, then $\mathcal{B}^{A}(e,r) \subseteq \mathcal{B}^{B}(e,Cr)$ for all $r>0$. By~\eqref{dilationball}, this amounts to
\[
\delta_{r}^{A} \mathcal{B}^{A}(e,1) \subseteq \delta_{r}^{B}  \mathcal{B}^{B}(e,C),
\]
which implies by~\eqref{measuredilation}, for $r>0$,
\[
r^{\tr(A)} \mu( \mathcal{B}^{A}(e,1)) \leq r^{\tr(B)} \mu(\mathcal{B}^{B}(e,C)).
\]
Thus, the function $r\mapsto r^{\tr(A) - \tr(B)}$ is bounded on $(0,\infty)$, and hence  $\tr(A) = \tr(B)$.
In particular, this shows $\det(\exp(A)) = \det(\exp(B))$, so that $\exp(A) = \exp(B)$ follows from \cite[Theorem 7.9]{cheshmavar2020classification}, provided
\begin{align} \label{eq:equivalent_norm_matrix}
 \sup_{k \in \mathbb{Z}} \| \exp(A)^{-k} \exp(B)^k \|_{\mathrm{GL}(\mathfrak{g})} < \infty.
\end{align}
Since
\begin{align*}
 \sup_{k \in \mathbb{Z}} \| \exp(A)^k \exp(B)^{-k} \|_{\mathrm{GL}(\mathfrak{g})}
 &\leq \sup_{r > 0} \| \exp(\ln(r)A) \exp(\ln(1/r) B) \|_{\mathrm{GL}(\mathfrak{g})} \\
 &= \sup_{r > 0} \sup_{x \in G \setminus \{e\}} \frac{\| \delta_r^A x \|}{\| \delta_r^B x\|}, \numberthis \label{eq:equiv_discrete_continuous}
\end{align*}
the desired conclusion will follow once we show that the quantity in~\eqref{eq:equiv_discrete_continuous} is finite.

In order to do this, note first that if $x \in G$ is such that $\|x\|=1$, then
\begin{align*}
\rho_{A}(\delta_{1/r}^{A}\delta_{r}^{B}x) &
= r^{-1}\rho_{A}(\delta_{r}^{B}x) \\
&\leq C r^{-1}\rho_{B}(\delta_{r}^{B}x)  = C \rho_{B}(x) \leq C \sup\{ \rho_{B}(z) \colon \|z\|=1\} \leq D,
\end{align*}
where the last supremum is finite because $\{z\in G\colon \|z\|=1\}$ is compact and $ \rho_{B}$ is continuous. In other words,
\[
\delta_{1/r}^{A}\delta_{r}^{B} \{ x\in G\colon \|x\|=1 \} \subseteq \mathcal{B}^{A}(e,D).
\]
Since $ \overline{\mathcal{B}^{A}(e,D)}$ is compact by~\cite[Lemma 1.4]{FollandStein}, there is $R>0$ such that
\[
\delta_{1/r}^{A}\delta_{r}^{B} \{ x\in G\colon \|x\|=1 \} \subseteq \mathcal{B}^{A}(e,D) \subseteq  \mathcal{B}(e,R),
\]
namely $\| \delta_{1/r}^{A}\delta_{r}^{B} x\| \leq R$ for all $r>0$ and $x\in G$ such that $\|x\|=1$. If now $x\in G$ is arbitrary, then $\delta^I_{\|x\|^{-1}} x \in \{ x\in G\colon \|x\|=1 \} $, and by~\eqref{euclnormcomm},
\[
\| \delta_{1/r}^{A}\delta_{r}^{B} x\| \leq R\|x\|
\]
for all $x\in G$ and $r>0$. This last inequality is equivalent to
\[
\| \delta_{1/r}^{A} x\| \leq R\|\delta_{1/r}^{B}x\|
\]
for all $x\in G$ and $r>0$, yielding
\[
\sup_{x\in G\setminus \{e\}} \sup_{r>0} \frac{\| \delta_{r}^{A} x\| }{\| \delta_{r}^{B} x\|} = \sup_{x\in G\setminus \{e\}} \sup_{r>0} \frac{\| \delta_{1/r}^{A} x\| }{\| \delta_{1/r}^{B} x\|} \leq R,
\]
which completes the proof.
\end{proof}

\section{Hardy spaces on $G$} \label{sec:hardy}
For $p \in (0,1]$ and $\Lambda \in \mathrm{GL}(\mathfrak{g})$, we consider the dilation of a function $f$ on $G$
\[
D^{\Lambda,p}_{t} f(x) := t^{\tr(\Lambda)/p} f(\delta^{\Lambda}_{t} (x)), \qquad x\in G, \, t>0.
\]
We shall equivalently write $ f_t^{\Lambda,p}$ for $D^{\Lambda,p}_{t} f$ and $ f_t^{\Lambda}$ for $ f_t^{\Lambda,1}$. Let us observe that since $ \delta^{\Lambda}_t= \delta^{c\Lambda}_{t^{1/c}}$ for all $t>0$, one has
\begin{equation}\label{dilphi}
 f^{\Lambda}_t(x) = t^{\tr(\Lambda)}f(\delta^{\Lambda}_{t} (x)) = (t^{1/c})^{ \tr(c\Lambda)} f(\delta^{c\Lambda}_{t^{1/c}} (x) ) = f^{c\Lambda}_{t^{1/c}}(x), \qquad x\in G.
\end{equation}

Let now $A \in \mathrm{GL}(\mathfrak{g})$ be admissible. Given a Schwartz function $\phi \in \mathcal{S}$ and a tempered distribution $f \in \mathcal{S}'$, the \emph{radial maximal function} $M^0_{\phi,A} f$ of $f$ (with respect to $A$ and $\phi$) is
\begin{align} \label{eq:radial}
 M^{0}_{\phi,A} f (x) = \sup_{t > 0} |f \ast \phi^A_{t} (x)|, \quad x \in G.
\end{align}

Suppose now that $\phi \in \Ss$ is a commutative approximate identity for $A$, that is, $\int_G \phi \; d\mu = 1$ and  $\phi_s^A \ast \phi_t^A = \phi_t^A \ast \phi_s^A$ for all $s, t > 0$, cf.~\cite{glowacki1986stable, dziubanski1992remark}.
For $p \in (0,1]$, we define the \emph{Hardy space} $H^p_A$ as
\[
H^p_A = \{ f \in \mathcal{S}' \colon M^0_{\phi,A} f \in L^p\},
\]
endowed with the quasi-norm
\begin{equation}\label{HpAnorm}
\| f\|_{H^{p}_{A}}^{p} :=  \| M^0_{\phi,A} f\|_{p}^{p}.
\end{equation}
Here and all throughout, the $L^{p}$ norms are taken with respect to the Haar measure $\mu$.

We comment on the choice of this definition, among all the others available, in Remark~\ref{rem:FS} below. We first show that with such definition $H^p_A$ is invariant under scaling of the dilation matrix $A$. This is a straightforward consequence of~\eqref{dilphi}, but we state it as a lemma for future reference.
\begin{lemma} \label{lem:scaling_invariant}
Let $A \in \mathrm{GL}(\mathfrak{g})$ be admissible and $c > 0$. Then $H^p_A = H^p_{cA}$ with equality of quasi-norms for all $p \in (0, 1]$.
\end{lemma}
\begin{proof}
It is enough to observe that~\eqref{dilphi} implies $M^0_{\phi,A} f = M^0_{\phi,cA} f$.
\end{proof}
We can now elaborate on this and on our definition of $H^{p}_{A}$.

\begin{remark}\label{rem:FS}
By Lemma~\ref{lem:scaling_invariant}, up to adjusting the dilation matrix $A$ if necessary, it may be assumed that the minimum eigenvalue of $A$ is $1$ without affecting the space $H^p_A$ or its norm. Therefore, though the minimum eigenvalue of $A$ being $1$ is a standing assumption in~\cite{FollandStein} which we do not make, several results therein are still valid in our setting. In particular, by combining Lemma~\ref{lem:scaling_invariant} and~\cite[Corollary~4.17]{FollandStein}, one sees that:
\begin{itemize}
\item[(a)] a different choice of the commutative approximate identity $\phi$ originates an equivalent quasi-norm~\eqref{HpAnorm} of $H^p_A$;
\item[(b)] by~\cite[Proposition~2.15]{FollandStein}, $H^{p}_{A}$ embeds continuously in $\mathcal{S}'$ for all $p\in (0,1]$;
\item[(c)] by~\cite[Proposition~2.16]{FollandStein}, the quasi-norm~\eqref{HpAnorm} induces a metric on $H^{p}_{A}$ which makes it a complete metric space.
\end{itemize}
Let us emphasize, however, that our definition of admissible matrix does not give rise to new spaces with respect to those of~\cite{FollandStein}, but rather allows (whenever needed) for a larger flexibility in the choice of the matrices which describe the same Hardy space. In view of all this, if one adheres strictly to the setting of~\cite{FollandStein}, i.e.,\  assumes that the minimum eigenvalue of an admissible matrix is $1$, then Theorem~\ref{thm:main} reads as follows: \emph{$H^p_A = H^p_B$ for some (equivalently, all) $p \in (0,1]$ if and only if $A = B$}.
 \end{remark}

 The following lemma will be used repeatedly.

\begin{lemma} \label{lem:equivalentnorms}
Suppose $p \in (0,1]$ and let $A,B \in \mathrm{GL}(\mathfrak{g})$ be admissible. If $H^p_A = H^p_B$, then their quasi-norms are equivalent.
\end{lemma}
\begin{proof}
By the discussion in Remark~\ref{rem:FS}, the maps
\[
(f, g) \mapsto \| f-g \|_{H^p_A}^p, \qquad (f, g) \mapsto \| f-g \|_{H^p_B}^p
\]
are invariant metrics making $H^p_A$ and $H^p_B$ respectively into complete metric spaces and $H^p_A = H^p_B \hookrightarrow \mathcal{S}'$, whence the map $\iota \colon H^p_A \to H^p_B, \; f \mapsto f$ is well defined and has a  closed graph. By the closed graph theorem $\iota$ is continuous, so that $\| f \|_{H^p_B} \lesssim \| f \|_{H^p_A}$ for all $f \in H^p_A = H^p_B$. The other inequality follows similarly.
\end{proof}

In the remainder of this section we discuss equivalent characterizations of $H^{p}_{A}$ which will be of use to prove Theorem~\ref{thm:main}. In view of Remark~\ref{rem:FS}, we shall assume that the minimum eigenvalue of $A$ is $1$. We begin with the following simple lemma.

\begin{lemma}\label{lemmaeta}
Suppose $A$ is an admissible matrix with minimum eigenvalue $1$. Then there exists $\gamma>0$  (depending on $A$) such that, for all homogeneous quasi-norms $\rho_{A}$ associated with $(\delta_{r}^{A})$, the following holds.
\begin{enumerate}
\item[(i)] For all $R>0$ there exist $c_{1},c_{2}>0$ (which depend on $A$, $\rho_{A}$ and $R$) satisfying, for all $x\in \overline{\mathcal{B}^{A}(e,R)}$,
\begin{align}\label{FSequiv}
c_{1}\|x\| \leq \rho_{A}(x) \leq c_{2}\|x\|^{\gamma}.
\end{align}
In particular,
\begin{equation}\label{inclusionballs}
\mathcal{B}\Big(e, \Big( \frac{ c_{1}}{c_{2}} R \Big)^{1/\gamma}\Big)\subseteq \mathcal{B}^{A}(e, c_{1}R) \subseteq \mathcal{B}(e,R).
\end{equation}

\item[(ii)] For all $R>0$ there exists $C>0$ (which depends on $A$, $\rho_{A}$ and $R$) such that, for all $x, y \in \overline{\mathcal{B}(e,R)}$,
\begin{align} \label{eq:triangle2}
 \| x y \| \leq C\big(\| x \|^{\gamma} + \| y \|^{\gamma}\big).
 \end{align}
\end{enumerate}
\end{lemma}

\begin{proof}
Assertion (i) can be proved in the exact same manner as~\cite[Proposition 1.5]{FollandStein}, whereas (ii) follows from a combination of (i) and~\eqref{eq:triangle}.
\end{proof}

\subsection{Atomic decompositions} \label{sec:atomic}
Assume that $A \in \mathrm{GL}(\mathfrak{g})$ is an admissible matrix whose minimum eigenvalue is $1$, and fix an associated homogeneous quasi-norm $\rho_{A}$. Denote by $v_1,\dots, v_n$ the eigenvalues of $A$, listed in increasing order (whence $v_{1}=1)$. Given a multiindex $I = (i_1, \dots, i_n) \in \mathbb{N}_0^n$, we define its \emph{isotropic degree} and \emph{homogeneous degree} (associated to $A$) respectively by
\[
 |I| = i_1 + \dots + i_n , \qquad d_A (I) = v_1 i_1 + \dots + v_n i_n.
\]
We denote by $\Delta_A$ the sub-semigroup of $\mathbb{R}$ generated by $\{0, v_1, \dots, v_n\}$, that is, $\Delta_A = \{ d_{A}(I) \colon I \in \mathbb{N}^n \}$. Note that $ |I| \leq d_A(I)$ and $\mathbb{N} \subseteq \Delta_A$ as $v_1 = 1$.

A function $P \colon G \to \mathbb{C}$ is called a \emph{polynomial} on $G$ if $P \circ \exp_G$ is a polynomial on $\mathfrak{g}$. Fix an eigenbasis $\{X_1, \dots , X_n \}$ for $A$, and let $\{X^*_1, \dots, X^*_n \}$ be the associated dual basis for $\mathfrak{g}^*$. For $j = 1, \dots, n$, set $\eta_{j, A} := X^*_j \circ \exp_G^{-1}$. Then each $\eta_{j,A}$ is a (homogeneous) polynomial on $G$, and every polynomial $P$ on $G$ can be written uniquely as
\begin{align} \label{eq:polynomial_homogeneous}
 P = \sum_{I \in \mathbb{N}_0^n} c_I \eta_A^I, \quad \eta_A^I := \eta_{1, A}^{i_1} \cdots \eta_{n,A}^{i_n},
\end{align}
where all but finitely many of the coefficients $c_I \in \mathbb{C}$ are zero. The \emph{homogeneous degree} (with respect to $A$) of a polynomial $P$ as in~\eqref{eq:polynomial_homogeneous} is defined to be $\max\{ d_A(I) \colon c_I \neq 0 \}$. The set of all polynomials of homogeneous degree at most $N \in \mathbb{N}$ with respect to $A$ is denoted by $\mathcal{P}^A_N$.

Suppose $p \in (0,1]$. An element $\alpha \in \Delta_A$ is said to be $p$-\emph{admissible for $A$} if $\alpha \geq \max \{ \alpha' \in \Delta_A \colon\alpha' \leq \tr(A) (p^{-1} - 1) \}$. A pair $(p,\alpha)$ is said to be \emph{admissible for $A$} if $\alpha$ is $
p$-admissible for $A$. Given such a pair $(p,\alpha)$, we say that a function $a \colon G \to \mathbb{C}$ is a $(p, \alpha)$-\emph{atom} associated to $A$ if it satisfies the conditions:
\begin{enumerate}
\item[(a1)] $\supp a \subseteq \mathcal{B}^{A}(x_{0},r)$ for some $x_{0}\in G$ and $r>0$,
\item[(a2)] $\|a\|_{\infty} \leq \mu(\mathcal{B}^{A}(x_{0},r))^{-\frac{1}{p}}$,
\item[(a3)] $\int_{G}a \cdot P  \, d\mu =0$ for all $P \in \mathcal{P}^A_\alpha$.
\end{enumerate}
We denote by $\mathscr{A}^{p}_{\alpha}(A)$ the family of all $(p,\alpha)$-atoms associated to $A$. 

If  $\alpha \in \Delta_A$ is $p$-admissible, then by~\cite[Theorem 3.30]{FollandStein}, the Hardy space $H^{p}_{A}$ coincides with the space of all tempered distributions $f \in \mathcal{S}'$ of the form
\[
f= \sum\nolimits_{j}\kappa_j a_{j}, \qquad \kappa_j \geq 0, \;\; (\kappa_j) \in \ell^{p}, \;\; a_{j}\in \mathscr{A}^{p}_{\alpha} (A),
\]
with the equivalence of quasi-norms
\begin{equation}\label{atomicnorm}
\|f\|_{H^{p}_{A}}^{p} \approx \|f\|_{H^{p}_{\alpha}(A)}^{p} := \inf \bigg\{ \|(\kappa_j)\|_{\ell^{p}}^{p} \colon f= \sum_{j}\kappa_j a_{j}, \;  a_{j}\in \mathscr{A}^p_{\alpha}(A)\bigg\}.
\end{equation}

Rather than the classical atoms satisfying conditions (a1)--(a3), we will make use of certain ``modified'' atoms.  Given an admissible pair $(p,\alpha)$ for $A$ and $R>0$, we shall call \emph{modified} $(p, \alpha,R)$-atom (associated to $A$) a function $a \colon G \to \mathbb{C}$ such that
\begin{enumerate}
 \item[(a1')] $\supp a \subseteq x_0 \delta_{\e^{k}}^{A} ( \mathcal{B}(e,R))$ for some $x_0 \in G$ and $k \in \mathbb{Z}$;
 \item[(a2')] $\| a \|_\infty \leq \mu(\delta^A_{\e^{k}} \mathcal{B}(e,R))^{ - \frac{1}{p}}$;
 \item[(a3)] $\int_G a \cdot P\, d\mu = 0$ for all $P \in \mathcal{P}^A_\alpha$.
\end{enumerate}

To show the relation between the above Hardy spaces and those defined by such modified atoms, we use Lemma~\ref{lemmaeta} to prove the following simple result.

\begin{lemma} \label{lem:atom}
Suppose $R>0$, let $A \in \mathrm{GL}(\mathfrak{g})$ be admissible with minimum eigenvalue $1$ and $(p,\alpha)$ be admissible for $A$. Then $H^{p}_A$ coincides with the atomic space defined in terms of modified $(p, \alpha,R)$-atoms associated to $A$, with equivalence of quasi-norms.
\end{lemma}

\begin{proof}
By the equivalence of quasi-norms~\eqref{atomicnorm}, it will be enough to show that any $(p,\alpha)$-atom associated to $A$ is a multiple of a modified $(p,\alpha,R)$ atom associated to $A$, and viceversa, with uniform constants depending only on $p$, $A$ and $R$. As the two proofs are essentially the same, we shall provide the details of the first one only.

Suppose that $a$ is a $(p,\alpha)$-atom associated to $A$, supported in a ball $\mathcal{B}^{A}(x_{0},r)$ for which the size condition (a2) holds. Then, for $k= [\ln( \frac{r}{ R c_{1}})]+1 \in \Z$,
\begin{align*}
\supp a & \subseteq x_{0} \delta_{\frac{r}{Rc_{1}}}^{A}\, \mathcal{B}^{A}(e,c_{1}R) \\
& \subseteq x_{0} \delta_{\e^{k}}^{A}\, \mathcal{B}^{A}(e,c_{1}R) \subseteq x_{0} \delta_{\e^{k}}^{A} \, \mathcal{B}(e,R), \end{align*}
the last inclusion by~\eqref{inclusionballs}. As for the size condition,
\[
\| a \|_\infty \leq \mu(x_{0}\delta^A_{\e^{k}} \mathcal{B}(e,R))^{- \frac{1}{p}} \bigg( \frac{\mu(\mathcal{B}^{A}(x_{0},r))}{ \mu(x_{0}\delta^A_{\e^{k}} \mathcal{B}(e,R))} \bigg)^{- \frac{1}{p}},
\]
where, by left invariance and~\eqref{measuredilation},
\[
\frac{\mu(\mathcal{B}^{A}(x_{0},r ))}{ \mu(x_{0}\delta^A_{\e^{k}} \mathcal{B}(e,R))}  = \frac{r^{\tr (A)}\mu(\mathcal{B}^{A}(e,1))}{ \e^{k\tr (A)}\mu(\mathcal{B}(e,R))} \geq  \bigg(\frac{r}{ \e\, r/(c_{1}R)} \bigg)^{\tr (A) } \frac{\mu(\mathcal{B}^{A}(e,1))}{ \mu(\mathcal{B}(e,R))} \geq C(A,R),
\]
and the conclusion follows.
\end{proof}

Lastly, we define a family of auxiliary functions that we will need in the proof of Theorem~\ref{thm:main}. For this, let $\{Y_1^*, \dots, Y_n^*\}$ be the dual basis for $\mathfrak{g}^*$ of the basis $\{Y_{1},\dots, Y_{n}\}$, and define $\eta_j := Y_j^* \circ \, \exp_G^{-1}$ for $j = 1, \dots, n$. As above, any polynomial $P$ on $G$ can be written uniquely as
\begin{align} \label{eq:polynomial_isotropic}
 \sum_{I \in \mathbb{N}_0^n} c_I \eta^I, \quad \eta^I := \eta^{i_1}_1 \dots \eta^{i_n}_n,
\end{align}
for finitely many nonzero coefficients $c_I \in \mathbb{C}$. The \emph{isotropic degree} of a polynomial $P$ as in~\eqref{eq:polynomial_isotropic} is $\max \{ |I| \colon c_I \neq 0 \}$, and we denote the set of all polynomials of isotropic degree at most $N \in \mathbb{N}$ by $\mathcal{P}_N$. Note that if $P \in \mathcal{P}_N$ and $\Lambda \in \mathrm{GL}(\mathfrak{g})$, then also $P \circ \delta^{\Lambda}_t \in \mathcal{P}_N$ for any $t>0$. Moreover, a change of basis and the fact that  $|I|\leq d_A(I)$ imply $\mathcal{P}^A_N \subseteq \mathcal{P}_N$ for all admissible matrices $A$ whose minimum eigenvalue is $1$.

Let now $A, B \in \mathrm{GL}(\mathfrak{g})$ be admissible with minimum eigenvalue $1$. For $p \in (0,1]$, choose $\alpha \in \mathbb{N}$ so large that $(p, \alpha)$ is admissible for both $A$ and $B$. Given $R>0$, consider the family $\mathcal{F}_{\alpha, p,R} (A,B)$ of functions $f \colon G \to \mathbb{C}$ such that
\begin{enumerate}
 \item[(f1)] $\supp f \subseteq x_0 \delta^A_{\e^{j_1}} \delta^B_{\e^{j_2}} \mathcal{B}(e,R)$ for some $x_0 \in G$ and $j_1, j_2 \in \mathbb{Z}$.
 \item[(f2)] $\| f \|_{\infty} \leq \e^{- j_1 \tr(A)/p} \e^{- j_2 \tr(B)/p}$.
 \item[(f3)] $\int_G f \cdot P \, d\mu = 0$ for all $P \in \mathcal{P}_{\alpha}$.
\end{enumerate}
The significance of this family for our purposes is provided by the following lemma.

\begin{lemma}\label{lemmaunifbound}
 Let $A, B \in \mathrm{GL}(\mathfrak{g})$ be admissible matrices with minimum eigenvalue~$1$. Suppose $H^p_A = H^p_B$ for some $p \in (0,1]$ and that $\alpha \in \mathbb{N}$  is such that $(p, \alpha)$ is admissible for both $A$ and $B$. Then for all $R>0$ there is a constant $C'> 0$ such that $\| f \|_{H^p_A} \leq C'$ for all $f \in \mathcal{F}_{\alpha, p,R}(A,B)$.
\end{lemma}
\begin{proof}
Suppose $f \in \mathcal{F}_{\alpha, p,R}(A,B)$. First, by~(f1), the function $D^{B,p}_{\e^{j_2}} D^{A,p}_{\e^{j_1}} f$ is supported in
 \[
  \delta_{\e^{-j_1}}^A \delta_{\e^{-j_2}}^B(x_0)\,  \mathcal{B}(e,R).
 \]
Second, by (f2),
\[
 \| D^{B, p}_{\e^{j_2}} D^{A,p}_{\e^{j_1}} f \|_{\infty} = \e^{\frac{j_2 \tr(B)}{p}} \e^{\frac{j_1 \tr(A)}{p}}  \| f \|_{\infty} \leq 1.
\]
Lastly, given $P \in \mathcal{P}_{\alpha}$, it follows by (f3) that
\begin{equation}\label{orthogonality}
 \int D^{B,p}_{\e^{j_2}} D^{A,p}_{\e^{j_1}} f \cdot P \, d\mu = \e^{(\frac{1}{p} - 1) j_2 \tr(B)} \e^{(\frac{1}{p} - 1) j_1 \tr(A)} \int  f \cdot P \circ \delta^B_{\e^{-j_2}} \circ \delta^A_{\e^{-j_1}} \, d\mu= 0.
\end{equation}
Hence, the function $\mu (\mathcal{B}(e,R))^{\frac{1}{p}} \cdot D^{B,p}_{\e^{j_2}} D^{A,p}_{\e^{j_1}} f$ is a modified $(p,\alpha,R)$-atom for both $A$ and $B$. By Lemma~\ref{lem:equivalentnorms} and the fact that $D_t^{A,p}$ (resp.\ $D_t^{B,p})$ is an isometry on $H^p_A$ (resp.\ $H^p_B$) gives
\[
 \| f \|_{H^p_A} = \| D_{\e^{j_1}}^{A,p} f \|_{H^p_A} \lesssim \| D_{\e^{j_1}}^{A,p} f \|_{H^p_B} = \| D_{\e^{j_2}}^{B,p} D_{\e^{j_1}}^{A,p} f \|_{H^p_B}
\]
with an implicit constant independent of $f$. Since $\mu (\mathcal{B}(e,R))^{\frac{1}{p}} \cdot D^{B,p}_{\e^{j_2}} D^{A,p}_{\e^{j_1}} f$ is a multiple of an ordinary $(p, \alpha)$-atom for $B$ (cf.~the proof of Lemma~\ref{lem:atom}), an application of \cite[Theorem 2.9]{FollandStein} yields
\[
 \| D_{\e^{j_2}}^{B,p} D_{\e^{j_1}}^{A,p} f \|_{H^p_B} \lesssim \mu (\mathcal{B}(e,R))^{-1/p} \lesssim 1
\]
where the constants are independent of $f$.
\end{proof}

\subsection{Grand maximal function}
Let $A$ be an admissible matrix with minimum eigenvalue $1$. Given $N \in \mathbb{N}$, the \emph{grand maximal function} associated to the radial maximal function~\eqref{eq:radial} is defined by
\[
 \mathcal{M}^{0}_{(N),A} f  \; = \sup_{ \substack{\phi \in \mathcal{S}, \| \phi \|_{(N)} \leq 1}} M_{\phi,A}^{0} f,
\]
where $\| \phi \|_{(N)}$ is a semi-norm on $\mathcal{S}$ associated to $A$. By \cite[Proposition 2.8 and Theorem 3.30]{FollandStein}, $H^p_A$ with $p \in (0,1]$ can be characterized as the space of $f \in \mathcal{S}'$ such that $\mathcal{M}_{(N),A}^{0} f \in L^{p}$, with the equivalence of semi-norms
\begin{equation}\label{grandmax}
 \| f \|_{H^{p}_{A}}^{p} \asymp \|\mathcal{M}_{(N),A}^{0} f   \|_{p}^{p},
\end{equation}
provided that $N \geq \min\{ N' \in \mathbb{N} \colon N' \geq \min \{ \alpha \in \Delta_A \colon\alpha > \tr(A) (p^{-1} - 1) \} \}$.

\section{Equivalence of Hardy spaces}\label{sec:equivhardy}
This section is devoted to proving Theorem~\ref{thm:main}. We start with the following lemma.

\begin{lemma} \label{lem:multiple_matrix}
Let $A, B \in \mathrm{GL}(\mathfrak{g})$ be admissible and $\eps = \tr(A) / \tr(B)$. Then the following assertions are equivalent:
\begin{enumerate}
 \item[(i)] $A = c B$ for some $c>0$;
 \item[(ii)] $\sup_{j \in \mathbb{Z}} \| \exp(A)^{-j} \exp(B)^{\lfloor \eps j \rfloor} \|_{\mathrm{GL}(\mathfrak{g})} $ is finite.
\end{enumerate}
\end{lemma}
\begin{proof}
Suppose that $A = cB$ for $c>0$. Then $\eps = c$, and hence, for $j \in \mathbb{Z}$,
\[ \exp(A)^{-j} \exp(B)^{\lfloor \eps j \rfloor} = \exp \big( (-j + \lfloor  jc \rfloor 1/c ) A) =: \exp(r_j A)
\]
where $-1/c \leq r_j \leq 0$. Therefore,
\[
 \sup_{j \in \mathbb{Z}} \| \exp(A)^{-j} \exp(B)^{\lfloor \eps j \rfloor} \|_{\mathrm{GL}(\mathfrak{g})}  \, \leq \sup_{- 1/c \leq r \leq 0} \| \exp(r A)  \|_{\mathrm{GL}(\mathfrak{g})} < \infty,
\]
the last fact as $r \mapsto \exp(rA)$ is continuous.

Conversely, suppose that $\exp(A)$ and $\exp(B)$ satisfy (ii). Define the matrix $B' = \frac{\tr(A)}{\tr(B)} B$. Then $\exp(A)$ and $\exp(B')$ satisfy (ii), and $\det(\exp(A)) = \det(\exp(B'))$. Therefore, an application of \cite[Theorem 7.9]{cheshmavar2020classification} yields that $\exp(A) = \exp(B')$. Since the matrix $\exp(A) = \exp(B')$ has only strictly positive eigenvalues, it has a unique real logarithm (cf.~\cite[Theorem 1.31]{higham2008functions}), whence $A = B'$.
\end{proof}

We are now ready to prove Theorem~\ref{thm:main}, which we restate for the reader's convenience. The overall method of the proof is inspired by that of~\cite[Theorem 10.5]{bownik2003anisotropic}.

\begin{theorem} \label{thm:classification_hardy}
 Let $A, B \in \mathrm{GL}(\mathfrak{g})$ be admissible. Then the following are equivalent:
\begin{enumerate}
\item[(i)] $H^p_A=H^p_B$ for some $p \in (0, 1]$;
\item[(ii)] $H^p_A = H^p_B$ for all $p \in (0,1]$;
\item[(iii)] $A= c B$ for some $c > 0$.
\end{enumerate}
\end{theorem}

\begin{proof}
By Lemma~\ref{lem:scaling_invariant}, (iii) implies (ii). The fact that (ii) implies (i) is immediate. Hence, it remains to show that (i) implies (iii).

Suppose that $H^p_A = H^p_B$ for some $p \in (0,1]$. By rescaling $A$ and $B$ if necessary, it may be assumed (cf. Lemma~\ref{lem:scaling_invariant}) that both $A$ and $B$ have minimum eigenvalue $1$. Then $A = c B$ if and only if $c = 1$.
We argue by contradiction, and suppose that $A \neq B$. Then, by Lemma~\ref{lem:multiple_matrix}, either
\[
 \limsup_{j \to +\infty} \big\| \exp(A)^j \exp(B)^{-\lfloor \varepsilon j \rfloor} \big\| = \infty, \quad \mbox{or} \quad  \limsup_{j\to -\infty} \big\| \exp(A)^j \exp(B)^{-\lfloor \varepsilon j \rfloor} \big\| =\infty,
\]
where in this proof we simply write $\| \cdot \| = \| \cdot \|_{\mathrm{GL}(\mathfrak{g})}$ for the operator norm.
Up to passing to a subsequence, we can assume that actually either
\begin{equation}\label{limbc}
 \lim_{j \to +\infty} \big\| \exp(A)^j \exp(B)^{-\lfloor \varepsilon j \rfloor} \big\| = \infty, \quad \mbox{or} \quad  \lim_{j\to -\infty} \big\| \exp(A)^j \exp(B)^{-\lfloor \varepsilon j \rfloor} \big\| =\infty.
\end{equation}
Note that, for fixed $j \in \mathbb{Z}$,
\[
 \big\| \exp(A)^j  \exp(B)^{-\lfloor \varepsilon j \rfloor -m} \big\|  \leq  \big\| \exp(A)^j \exp(B)^{-\lfloor \varepsilon j \rfloor} \big\| \big\| \exp(B)^{-m}\big\|  \to 0
\]
as $m \to \infty$, since all eigenvalues of $\exp(B)$ are strictly greater than $1$. Therefore, there exists the smallest integer $m$ (and we call it $d_{j}$) such that
\[
 \big\| \exp(A)^j  \exp(B)^{-\lfloor \varepsilon j \rfloor -m}\big\| \leq 1.
\]
Since, by definition of $d_{j}$, the above inequality fails for $d_{j}-1$,
\begin{align*}
 \big\| \exp(A)^j\exp(B)^{-\lfloor \varepsilon j \rfloor -d_{j}}  \big\|
 & \geq \big\|\exp(A)^j  \exp(B)^{-\lfloor \varepsilon j \rfloor -d_{j}+1 } \big\| \| \exp(B)\|^{-1} \\
 & \geq \| \exp(B)\|^{-1},
\end{align*}
whence, for all $j\in \Z$,
\begin{equation}\label{doublecontrol}
\| \exp(B)\|^{-1}  \leq  \big\| \exp(A)^j \exp(B)^{-\lfloor \varepsilon j \rfloor -d_{j}} \big\|  \leq 1.
\end{equation}
If $(d_{j})$ was a bounded sequence, then~\eqref{limbc} could not hold; therefore, either $d_{j} \to +\infty$ or $d_{j} \to -\infty$  as $j\to + \infty$ {or $j\to -\infty$}. We consider the first case only, the others being analogous.

The remainder of the proof is split into three steps.

\bigskip

\noindent \textbf{Step 1.} (\emph{Auxiliary functions}). We construct a sequence of functions satisfying properties (f1)--(f3) considered in Section~\ref{sec:atomic}. For this, let $X \in\mathfrak{g}$ be such that $\|X \|=1$, and define 
\[
x_{1}= \exp_G(  X).
\]
Notice that $\|x_{1}\|= 1$. Choose  $\epsilon, \theta \in (0,1)$ small enough so that the balls $ \mathcal{B}(x_{1}, \epsilon)$ and $\mathcal{B}(e, \theta)$ are disjoint. In addition, fix $R>0$ such that $\mathcal{B}(x_{1}, \epsilon) \subseteq \mathcal{B}(e,R)$. Even though $\|\cdot \|$ is not a norm, all of these choices are possible by means of Lemma~\ref{lemmaeta} (applied to $A$ or $B$, after an associated homogeneous semi-norm on $G$ is chosen). Indeed, if $y \in \mathcal{B}(x_1, \epsilon)$, then \eqref{eq:triangle2} yields a constant $c > 0$ (only depending on $A$ or $B$) such that
\[
\| y \| = \| x_1 \cdot x_1^{-1} \cdot y \| \geq \big( c \| x_1 \| - \| x_1^{-1} y \|^{\gamma} \big)^{1/\gamma} \geq \big( c - \epsilon)^{1/\gamma} > \theta,
\]
i.e., $y \in \mathcal{B}(e, \theta)^c$, provided $\epsilon, \theta \in (0,1)$ are sufficiently small.
With such choices, we proceed to construct a function $a_{0}$ satisfying (f1)--(f3) in Section~\ref{sec:atomic} with $x_0 = e$ and $j_1 = j_2 = 0$.

Let $\alpha \in \mathbb{N}$ be so large that $(p, \alpha)$ is admissible for both $A$ and $B$. Set $n_{\alpha} := \# \{ I \in \mathbb{N}_0^n \colon|I| \leq \alpha \}$, and define the map
\[
T \colon L^{\infty} (\mathcal{B}(e, \theta)) \to \mathbb{R}^{n_{\alpha}}, \quad f \mapsto \bigg( \int_{\mathcal{B}(e, \theta)} f(x) \eta^I (x) \; d\mu(x) \bigg)_{|I| \leq \alpha}.
\]
Then $T$ is surjective, and hence  defining $v \in \mathbb{R}^{n_{\alpha}}$ by $v_{I} := - \int_{\mathcal{B}(x_1, \epsilon)} \eta^I \, d\mu$, there exists $f \in L^{\infty} (\mathcal{B}(e, \theta))$ such that $T(f) = v$.
Let $\tilde a_0 \colon G \to \mathbb{C}$ be defined by
\[
\tilde a_0(x)=
\begin{cases}
f(x)  \quad & \mbox{if } x \in \mathcal{B}(e,\theta), \\
1 \quad & \mbox{if } x \in \mathcal{B}(x_{1},\epsilon) ,\\
0 & \mbox{if } x\notin \mathcal{B}(e,\theta) \cup \mathcal{B}(x_{1},\epsilon).
\end{cases}
\]
Then $\supp \tilde a_{0} \subseteq  \mathcal{B}(e,R)$ and $\int_G \tilde a_0 \cdot \eta^I \, d\mu = 0$ for all $|I|\leq \alpha$. Choose now $\omega_{0}\in (0, 1]$ such that, if $a_0:=\omega_{0}\, \tilde a_{0}$, then $\|a_{0}\|_{\infty}\leq 1$. Then $a_0 \in \mathcal{F}_{\alpha, p,R}(A,B)$.

We now construct suitable dilations of $a_0$. Define $Q_{j}$ to be the matrix such that
\begin{equation}\label{defQj}
\exp(Q_{j}) =  \exp(A)^j\exp(B)^{-\lfloor \varepsilon j \rfloor -d_{j}}, \qquad j\geq 1,
\end{equation}
by using the Baker--Campbell--Hausdorff formula. Then $\exp(Q_j)$ is an automorphism of $\mathfrak{g}$ as it is composition of automorphisms. As such, $\delta_{e}^{Q_{j}}$ and $\delta_{1/e}^{Q_{j}}$ are automorphisms of $G$.

Then pick $Z_{j}\in \mathfrak{g}$ such that $\|Z_{j}\|=1$ and
\begin{equation}\label{defcj}
\|\exp(Q_{j})Z_{j}\|= \| \exp(Q_{j})\|=:\tau_j.
\end{equation}
By~\eqref{doublecontrol} then
\begin{equation}\label{doublecontrolc}
\| \exp(B)\|^{-1}  \leq \tau_j \leq 1,
\end{equation}
and by taking the determinants in~\eqref{defQj}
\begin{equation}\label{traceQj}
\tr(Q_{j}) = j\tr(A) - (\lfloor \varepsilon j \rfloor +d_{j}) \tr(B).
\end{equation}
In addition to $Q_j$, choose a matrix $U_{j}$ such that $\exp(U_{j})$ is unitary and $\exp(U_{j})X = Z_{j}$, and define $z_{j}= \exp_G(Z_{j})$. Then define
$
a_{j} := D_{\e^{-1}}^{Q_{j},p} D_{\e^{-1}}^{U_{j},p}a_{0}$ for $j \in \mathbb{N}$.

We claim that $a_j \in \mathcal{F}_{\alpha, p,R} (A,B)$ for all $j\in \N$. For this, observe first that since $\exp(U_{j})$ is unitary, it holds that $\delta_{\e}^{U_{j}}\mathcal{B}(e,r) \subseteq \mathcal{B}(e,r) $ for all $r>0$. Moreover,
\begin{equation}\label{Uj34}
\delta_{\e}^{U_{j}}x_{1} = \exp_G( \exp(U_{j})X) = \exp_G( Z_{j}) =  z_{j}.
\end{equation}
Since
\[
\delta_{\e}^{Q_{j}} = \delta_{\e^{j}}^{A} \delta_{\e^{-\lfloor \varepsilon j \rfloor -d_{j}}}^{B},
\]
it holds that $\delta_{\e}^{Q_{j}}\mathcal{B}(e,R) \subseteq \mathcal{B}(e,\tau_jR)$ by definition~\eqref{defcj} of $\tau_j$, and
\[
\supp a_{j}  \subseteq \delta_{\e}^{Q_{j}} \mathcal{B}(e,R) =\delta_{\e^{j}}^{A}  \delta_{\e^{-\lfloor \varepsilon j \rfloor -d_{j}}}^{B} \mathcal{B}(e,R),
\]
showing (f1). Moreover, $\supp a_{j} \subseteq \mathcal{B}(e,R)$ for all $j$'s.  More precisely, by~\eqref{Uj34},
\[
\delta_{\e}^{Q_{j}} \delta_{\e}^{U_{j}} \mathcal{B}(x_{1},\epsilon) = \delta_{\e}^{Q_{j}}  \mathcal{B}(z_{j} ,\epsilon),
\]
and thus
\begin{equation}\label{suppaj}
\supp a_{j}  \subseteq \delta_{\e}^{Q_{j}} ( \mathcal{B}(e,\theta) \cup \mathcal{B}(x_{1},\epsilon) ) \subseteq \mathcal{B}(e,\tau_j\theta) \cup  \delta_{\e}^{Q_{j}}  \mathcal{B}(z_{j} ,\epsilon).
\end{equation}
For (f2), note that, by~\eqref{traceQj} and $\tr(U_{j})=0$,
\begin{align*}
\| a_j \|_{\infty} = \e^{-\frac{\tr(Q_{j})}{p} - \frac{\tr(U_{j})}{p}} \| a_0 \|_{\infty} \leq \e^{-j\frac{\tr(A)}{p} + (\lfloor \varepsilon j \rfloor +d_{j})\frac{ \tr(B)}{p}},
\end{align*}
as required. In addition, we note that
\begin{equation}\label{gammaj}
\begin{split}
a_{j}(\delta_{\e}^{Q_{j}}  \mathcal{B}(z_{j} ,\epsilon))
&= \e^{-\frac{\tr(Q_{j})}{p} - \frac{\tr(U_{j})}{p}} a_{0} (  \mathcal{B}(x_{1} ,\epsilon))\\
& = \e^{-j\frac{\tr(A)}{p} + (\lfloor \varepsilon j \rfloor +d_{j})\frac{ \tr(B)}{p}}\omega_{0} =: \omega_{j}.
\end{split}
\end{equation} 
Lastly, arguing as in~\eqref{orthogonality}, one sees that also~(f3) holds.

\bigskip

\noindent \textbf{Step 2.} (\emph{Case $p = 1$}).
Suppose that $p=1$. By construction, see~\eqref{defcj} and~\eqref{doublecontrolc}, we have that $\| \exp(B) \|^{-1} \leq \| \exp (Q_j) \| \leq 1$ and $\| Z_j \| = 1$ for $j \in \mathbb{N}$. Hence, by passing to a subsequence if necessary, it may be assumed that $\exp(Q_j) \to Q'$ for some matrix $Q' : \mathfrak{g} \to \mathfrak{g}$ satisfying $\| \exp(B) \|^{-1} \leq \| Q' \| \leq 1$, and that $Z_j \to Z'$ for some $Z' \in \mathfrak{g}$ with $\|Z'\| =1$. In addition, it may be assumed that $\exp(U_j) \to U'$ for some unitary matrix $U' \in \mathrm{GL}(\mathfrak{g})$. Since $\varepsilon = \tr(A) / \tr(B)$ and $d_j \to +\infty$, it follows that
\begin{align*}
 |\det (\exp(Q_j)) | &= \e^{\tr(Q_j)} = \e^{j \tr(A) - ( \lfloor \varepsilon j \rfloor + d_j ) \tr(B)} \leq \e^{(1-d_j) \tr(B)}  \to 0,
\end{align*}
as $j \to \infty$. Hence, $|\det (Q') | = 0$, and,  in particular, $Q'$ is not surjective.

Next, we show that the sequence $(a_j)_{j \in \mathbb{N}}$ of functions $a_j \in \mathcal{F}_{\alpha, p, R} (A, B)$ constructed in Step 1 converges to a nonzero regular Borel measure $a$. For this, let $\varphi \in C_b (G)$ be arbitrary. Then a direct calculation entails
\begin{align*}
 \int_G a_j (x) \varphi(x) \; d\mu(x)
 &= \e^{- \tr(Q_j)} \int_G (D_{\e^{-1}}^{U_j} a_0)(\delta_{\e^{-1}}^{Q_j} (x)) \varphi (\delta_\e^{Q_j} \delta_{\e^{-1}}^{Q_j} (x)) \; d\mu(x) \\
 &= \int_G (D_{\e^{-1}}^{U_j} a_0)(y) \varphi (\delta_\e^{Q_j} (y)) \; d\mu(y) \\
 &= \int_G a_0(z) \varphi (\delta_\e^{Q_j} \delta_\e^{U_j} (z)) \; d\mu(z).
 \end{align*}
 Using that $\varphi (\delta_\e^{Q_j} \delta_\e^{U_j} (z)) \to \varphi (\exp_G \circ \; Q' \circ U' \circ \exp_G^{-1} (z))$, together with the dominated convergence theorem, it follows therefore that
 \begin{align*}
 \int_G a_j (x) \varphi(x) \; d\mu(x)  &\to \int_G a_0(z) \varphi (\exp_G \circ \; Q' \circ U' \circ \exp_G^{-1} (z)) \; d\mu(z) \\
 &=: \int_G \varphi (x) \; da (x)
\end{align*}
for a unique regular Borel measure $a$ on $G$. Since $Q'$ is not surjective, it follows that $\supp a \neq G$, and thus $a$ is singular with respect to the Haar measure $\mu$ on $G$. 

We shall show that $a \neq 0$. Given some $z \in \mathcal{B}(e,\theta)$, set $z' = \exp_G \circ \;Q' \circ U' \circ \exp_G^{-1} (z)$, and set $x_1' = \exp_G \circ \;Q' \circ U' \circ \exp_G^{-1} (x_1)$. Then
\begin{align*}
 \| z' \| = \| Q' U' \exp_G^{-1} (z) \| \leq \| Q' \| \| U' \exp_G^{-1} (z) \| \leq \theta  \|Q' \| ,
\end{align*}
and, using \eqref{defcj} and \eqref{Uj34},
\begin{align*}
 \| x'_1 \| &=   \lim_{j \to \infty} \|\exp(Q_j) \exp(U_j) X \|
 = \lim_{j \to \infty} \| \exp(Q_j) Z_j \|
 = \lim_{j \to \infty} \| \exp(Q_j ) \| \\
 &= \| Q' \|.
\end{align*}
Therefore, an application of Lemma~\ref{lemmaeta} (with $R = 1$) yields a constant $c>0$ such that
\begin{align*}
\|(x'_1)^{-1} z'  \| \geq \big(c \| x_1' \| - \|z' \|^{\gamma} \big)^{1/\gamma} \geq \big( c \| Q' \| - \theta  \| Q'\|^{\gamma} \big)^{1/\gamma}.
\end{align*}
Hence, by decreasing $\theta \in (0,1]$ if necessary, there exists $\delta > 0$ such that $\| (x_1')^{-1} z'  \| > \delta$, that is, $z' \notin \mathcal{B}(x_1', \delta)$. Choose now a non-negative continuous function $\varphi$ satisfying $\supp \varphi \subseteq \mathcal{B} (x_1', \delta)$ and $\varphi(x_1') = 1$. Then $\varphi(\exp_G \circ \;Q' \circ U' \circ \exp_G^{-1} (z)) = 0$ for any $z \in \mathcal{B}(e, \theta)$, and, by construction of $a_0$,
\begin{align*}
 \int_G \varphi (x) \; da (x) &= \int_G a_0 (z) \varphi(\exp_G \circ \; Q' \circ U' \circ \exp_G^{-1} (z)) \; d\mu(z) \\
 &= \int_{\mathcal{B}(e, \theta)} a_0 (z) \varphi(\exp_G \circ \; Q' \circ U' \circ \exp_G^{-1} (z)) \; d\mu(z) \\
 & \quad \quad \quad + \omega_0 \int_{\mathcal{B}(x_1, \epsilon)} \varphi( \exp_G \circ \; Q' \circ U' \circ \exp_G^{-1} (z)) \; d\mu(z) \\
 &= \omega_0 \int_{\mathcal{B}(x_1, \epsilon)} \varphi( \exp_G \circ \; Q' \circ U' \circ \exp_G^{-1} (z)) \; d\mu(z)> 0,
\end{align*}
where the inequality follows from the fact that $\varphi\geq 0$ is continuous and $\varphi(x_1') = 1$. This shows that $a$ is nonzero, whence in particular $a\notin L^1$.

On the other hand, by Fatou's lemma and the grand maximal characterization~\eqref{grandmax},
\begin{align*}
\|a\|_{H^{1}_{A}} \approx \int \mathcal{M}_{(N),A}^{0} a \, \dd \mu
&= \int \lim_{i\to \infty}\mathcal{M}_{(N),A}^{0} a_{j_{i}} \, \dd \mu \\
& \leq \liminf_{i\to \infty} \int \mathcal{M}_{(N),A}^{0} a_{j_{i}} \, \dd \mu \leq  C \liminf_{i\to \infty} \|a_{j_{i}}\|_{H^{1}_{A}} \leq C',
\end{align*}
where the last bound follows from Lemma~\ref{lemmaunifbound}. Thus $a\in H^{1}_{A}$. This contradicts that $a \notin L^1$, and completes the proof for $p = 1$.

\bigskip

\noindent \textbf{Step 3.} (\emph{Case $p < 1$}).
Suppose that $p<1$ and set $\sigma := 1/\gamma^{2}$, where $\gamma  > 0$ is that of  Lemma~\ref{lemmaeta}. Pick $\phi \in \Ss$ such that $\phi = 1$ on $\mathcal{B}(e, \epsilon_{1} \|\exp(B)\|^{-\sigma})$ and $\supp \phi \subseteq \mathcal{B}(e, \epsilon_{2} \|\exp(B)\|^{-\sigma})$ for some $0<\epsilon_{1}<\epsilon_{2} < 1$ to be determined.

For all $z\in G$, by~\eqref{suppaj} and since $\phi(x^{-1}z) = 0$ if  $x \notin z\mathcal{B}(e, \epsilon_{2}\|\exp(B)\|^{-\sigma})$,
\[
M_{\phi,A}^{0}a_{j}(z) \geq \left| \int_{z\mathcal{B}(e, \epsilon_{2}\|\exp(B)\|^{-\sigma}) \cap ( \mathcal{B}(e,\tau_j\theta) \cup  \delta_{\e}^{Q_{j}}  \mathcal{B}(z_{j} ,\epsilon) )} a_{j}(x) \phi(x^{-1}z) \, \dd \mu (x) \right|.
\]
Suppose now that $z\in \mathcal{B}( \delta_{\e}^{Q_{j}} (z_{j}), \beta \|\exp(B)\|^{-\sigma})$ for some $\beta \in (0,1)$. Then, by~\eqref{eq:triangle2},
\begin{align*}
z\mathcal{B}(e, \epsilon_{2}\|\exp(B)\|^{-\sigma})
&\subseteq  \delta_{\e}^{Q_{j}} (z_{j})  \cdot \mathcal{B}( e, \beta \|\exp(B)\|^{-\sigma}) \mathcal{B}(e, \epsilon_{2}\|\exp(B)\|^{-\sigma})\\
& \subseteq   \delta_{\e}^{Q_{j}}(z_{j}) \cdot \mathcal{B}(e, c_{\beta,\epsilon_{2}} \|\exp(B)\|^{-1/\gamma}),
\end{align*}
with $c_{\beta,\epsilon_{2}}$ small if $\beta$ and $\epsilon_{2}$ are small. Thus, if $x\in z\mathcal{B}(e, \epsilon_{2}\|\exp(B)\|^{-\sigma}) $, then $ \| (\delta_{\e}^{Q_{j}}z_{j})^{-1} \cdot x \| \leq c_{\beta,\epsilon_{2}} \|\exp(B)\|^{-1/\gamma}$, so that by~\eqref{eq:triangle2},~\eqref{defcj} and~\eqref{doublecontrolc} there exists $c>0$ (independent of $j \in \mathbb{N}$) such that
\begin{align*}
\|x\|
&= \| \delta_{\e}^{Q_{j}}z_{j} \cdot (\delta_{\e}^{Q_{j}}z_{j})^{-1} \cdot x  \|  \\
&  \geq  \Big(c \|\delta_{\e}^{Q_{j}}z_{j} \| - \|(\delta_{\e}^{Q_{j}}z_{j})^{-1}  \cdot x \|^{\gamma} \Big)^{1/\gamma}  \geq  ( c\,  \tau_j - c_{\beta,\epsilon_{2}} \tau_j)^{1/\gamma} > \theta \tau_j,
\end{align*}
i.e.\ $x \notin \mathcal{B}(e,\theta \tau_j)$, provided $\beta$, $\epsilon_{2}$ and $\theta$ are small enough. Observe that $\tau_j$ is bounded above and below away from $0$ by~\eqref{doublecontrolc}, whence $\tau_j^{1/\gamma} \approx \tau_j$. The above proves that
\begin{align*}
z\mathcal{B}(e, \epsilon_{2}\|\exp(B)\|^{-\sigma}) &\cap ( \mathcal{B}(e,\tau_j\theta) \cup  \delta_{\e}^{Q_{j}}  \mathcal{B}(z_{j} ,\epsilon) )\\
&  =  z\mathcal{B}(e, \epsilon_{2}\|\exp(B)\|^{-\sigma})  \cap \delta_{\e}^{Q_{j}}(z_{j}) \cdot  \delta_{\e}^{Q_{j}} \mathcal{B}(e ,\epsilon),
\end{align*}
where we used that $ \delta_{\e}^{Q_{j}}  \mathcal{B}(z_{j} ,\epsilon) = \delta_{\e}^{Q_{j}}(z_{j}) \cdot  \delta_{\e}^{Q_{j}} \mathcal{B}(e ,\epsilon)$. Therefore, by the above,~\eqref{gammaj} and  since $(\delta_{\e}^{Q_{j}}(z_{j}))^{-1}z\in \mathcal{B}( e, \beta \|\exp(B)\|^{-\sigma})$, if $\beta$ and $\epsilon_{1}$ are small enough then
\begin{align}\label{inclballs}
M_{\phi,A}^{0}a_{j}(z)  
& \geq \omega_{j}\, \mu(\mathcal{B}(z, \epsilon_{2}\|\exp(B)\|^{-\sigma})  \cap \delta_{\e}^{Q_{j}}(z_{j}) \cdot  \delta_{\e}^{Q_{j}} \mathcal{B}(e,\epsilon)) \nonumber \\
& =  \omega_{j}\, \mu((\delta_{\e}^{Q_{j}} (z_{j}))^{-1}z  \cdot \mathcal{B}(e, \epsilon_{2}\|\exp(B)\|^{-\sigma})  \cap \delta_{\e}^{Q_{j}}  \mathcal{B}(e,\epsilon))\nonumber  \\
& \geq \omega_{j}\, \mu( \mathcal{B}(e,\epsilon_{ 1} \|\exp(B)\|^{-\sigma}) \cap \delta_{\e}^{Q_{j}}  \mathcal{B}(e,\epsilon))\nonumber \\
& = \omega_{j} \, \e^{\tr(Q_{j})} \mu( \delta_{1/ \e}^{Q_{j}}  \mathcal{B}(e,\epsilon_{ 1} \|\exp(B)\|^{-\sigma}) \cap \mathcal{B}(e,\epsilon)).
\end{align}
Observe now that $ \delta_{1/ \e}^{Q_{j}}  \mathcal{B}(e,\epsilon_{ 1} \|\exp(B)\|^{-\sigma})  \supseteq \mathcal{B}(e,\tau_j^{-1} \epsilon_{ 1} \|\exp(B)\|^{-\sigma})$. Thus, if $\epsilon_{ 1}$ is small enough so that $\tau_j^{-1} \epsilon_{ 1} \|\exp(B)\|^{-\sigma} < \epsilon$, then by~\eqref{inclballs}
\begin{align*}
M_{\phi,A}^{0}a_{j}(z)   & \geq \omega_{j}\, \e^{\tr(Q_{j})} \mu( \mathcal{B}(e,\tau_j^{-1} \epsilon_{ 1} \|\exp(B)\|^{-\sigma}))\\
& = \omega_{j}\, \e^{\tr(Q_{j})} \mu( \delta_{\tau_j^{-1} \epsilon_{ 1} \|\exp(B)\|^{-\sigma}}^{I}\mathcal{B}(e, 1))  \geq c \, \omega_{j}\, \e^{\tr(Q_{j})}.
\end{align*}
Thus, by~\eqref{traceQj} and~\eqref{gammaj},
\begin{align*}
M_{\varphi}^{0,A}a_{j}(z)
\geq c\, \e^{(\frac{1}{p}-1) [\lfloor \varepsilon j \rfloor \tr(B)  - j\tr(A)]} \e^{(\frac{1}{p}-1) \tr(B)d_{j}} .
\end{align*}
Since
\begin{align*}
 0 = \varepsilon j  \tr(B)  -j\tr(A) \geq   \lfloor \varepsilon j \rfloor \tr(B)  -j\tr(A) \geq (\varepsilon j -1) \tr(B)  -j\tr(A) = -\tr(B)
 \end{align*}
 for all $j\in \N$, we conclude
 \[
 M_{\phi,A}^{0}a_{j}(z)  \geq c' \e^{(\frac{1}{p}-1) \tr(B)d_{j}},
 \]
 from which it follows that 
\begin{align*}
 \int_{G} |M_{\phi,A}^{0}a_{j}|^{p} \, \dd \mu
 &  \geq \int_{\mathcal{B}( \delta_{\e}^{Q_{j}}z_{j}, \beta \|\exp(B)\|^{-\sigma})} |M_{\phi,A}^{0}a_{j} |^{p} \, \dd \mu\\
 & \geq C \mu(\mathcal{B}( e, \beta \|\exp(B)\|^{-\sigma}))\, \e^{(\frac{1}{p}-1) \tr(B)d_{j}} \to \infty
 \end{align*}
as $j \to \infty$, which is a contradiction by Lemma~\ref{lemmaunifbound} and the grand maximal function characterization~\eqref{grandmax} of the Hardy space seminorm. This completes the proof.
\end{proof}

Lastly, we have the following simple consequence on equivalence of the dual of the Hardy spaces $H^1_A$ and $H^1_B$. These spaces can be identified
with $BMO$ spaces; see \cite{FollandStein, bownik2007duals} for definitions and precise details. In particular, $(H^1_A)^*, (H^1_B)^* \hookrightarrow \mathcal{S}' / \mathcal{P}_0$ by \cite[Proposition 5.9]{FollandStein}. 

\begin{corollary}
Let $A, B \in \mathrm{GL}(\mathfrak{g})$ be admissible. Then $(H^1_A)^* = (H^1_B)^*$ if and only if $A = c B$ for some $c > 0$.\end{corollary}
\begin{proof}
Suppose that $(H^1_A)^* = (H^1_B)^*$. By arguing as in Lemma~\ref{lem:equivalentnorms},
one sees that $\| f \|_{(H^1_A)^*} \asymp \| f \|_{(H^1_A)^*}$ holds for all $f \in (H^1_A)^* = (H^1_B)^*$. By duality then
\[
 \| h \|_{H^1_A} = \sup_{\substack { f \in (H^1_A)^* \\ \| f \|_{(H^1_A)^*} \leq 1}} |f(h)| \asymp \sup_{\substack { f \in (H^1_B)^* \\ \| f \|_{(H^1_B)^*} \leq 1}} |f(h)| = \| h \|_{H^1_B}
\]
for all $h \in H^1_A = H^1_B$. An application of Theorem~\ref{thm:classification_hardy} therefore yields $A = c B$.

Conversely, if $A = c B$, then $H^1_A = H^1_B$ by Proposition~\ref{lem:scaling_invariant}, and the result follows by duality.
\end{proof}

\subsection*{Acknowledgements} The authors wish to thank M.\ Bownik for clarifying comments on \cite{bownik2020pde}, and F.\ Voigtlaender for a helpful discussion on a part of the proof of Theorem~\ref{thm:main}.

\end{document}